\documentclass[11pt]{article}

\usepackage[a4paper,margin=1in]{geometry}
\usepackage{amsmath,amssymb,amsthm}
\usepackage[hidelinks]{hyperref}
\usepackage{xcolor}
\usepackage{ulem}

\newtheorem{theorem}{Theorem}[section]
\newtheorem{lemma}[theorem]{Lemma}
\newtheorem{corollary}[theorem]{Corollary}
\theoremstyle{remark}
\newtheorem{remark}[theorem]{Remark}

\newcommand{\T}{\mathbb{T}}

\newcommand{\Z}{\mathbb{Z}}
\newcommand{\R}{\mathbb{R}}
\newcommand{\Haus}{\mathcal{H}}
\newcommand{\Fp}{F_p}
\newcommand{\Lp}{L_p}

\newcommand{\bad}{\mathbf{Bad}}
\newcommand{\mpc}{m_p}
\newcommand{\dist}{\operatorname{dist}}
\newcommand{\Iq}{\mathbb{I}_q}
\newcommand{\Et}{E_q^{(t)}}
\newcommand{\mt}{m_t}

\title{Positive Logarithmic Hausdorff Measures of Exceptional Sets for the $p$-adic and $t$-adic Littlewood Conjectures}
\author{Dzmitry BADZIAHIN\textsuperscript{1} \& Volodymyr PAVLENKOV\textsuperscript{2}\& Evgeniy ZORIN\textsuperscript{3}}
\date{}

\begin{document}

\maketitle

\let\thefootnote\relax\footnotetext{\textsuperscript{1}Department of Mathematics, University of Sydney,
Australia, {\texttt{dzmitry.badziahin@sydney.edu.au }}.}
\let\thefootnote\relax\footnotetext{\textsuperscript{2}Department of Mathematics, University of York,
Heslington, York, YO10 5DD, England, {\texttt{volodymr.pavlenkov@york.ac.uk}}.}
\let\thefootnote\relax\footnotetext{\textsuperscript{3} Department of Mathematics, University of York,
Heslington, York, YO10 5DD, England, \texttt{evgeniy.zorin@york.ac.uk}.\\}

\begin{abstract}
We prove that if the exceptional set $E_p$ for the $p$-adic Littlewood
conjecture is non-empty, then its logarithmic Hausdorff dimension is at least
one.  More precisely, whenever $E_p$ is non-empty, it has positive Hausdorff
measure with respect to the gauge function
$
h(r)=\frac{1}{\log(1/r)}.
$
In particular, every non-empty $E_p$ has the cardinality of the continuum.

We obtain stronger conclusions for the $t$-adic Littlewood conjecture over a
finite field $\mathbb F_q$.  For every prime power $q$, non-emptiness of the
exceptional set $E_q^{(t)}$ implies that its $1/\log(1/r)$-Hausdorff measure is
infinite.  Moreover,  when $q$ is odd, we refine the recent construction of Lai and
Sprang~\cite{LaiSprang2026} and prove that
\[
    \mathcal H^{h_{A_q}}(E_q^{(t)})=\infty,
\]
where
\[
    h_{A_q}(r)=\frac{1}{(\log(1/r))^{A_q}},
    \qquad
    A_q=\frac{q-1}{2}\log_2(q-1).
\]
In characteristic two, the corresponding conclusion with exponent one
remains conditional on the existence of a counterexample.
\end{abstract}

\section{Introduction}

Let $p$ be a prime, let $|\cdot|_p$ denote the $p$-adic absolute value
normalised by $|p|_p=p^{-1}$, and, for any $t\in\R$, write $\|t\|=\dist(t,\Z)$.  The $p$-adic
Littlewood conjecture, introduced by de Mathan and Teuli\'e
\cite{deMathanTeulie2004}, states that
\begin{equation} \label{eq:plc}
    \liminf_{q\to\infty}q|q|_p\|qx\|=0
    \qquad\text{for every $x\in\mathbb R$.}
\end{equation}
This conjecture was formulated as a $p$-adic analogue of the classical
Littlewood conjecture, which asserts that
\begin{equation} \label{eq:classical-lc}
    \liminf_{q\to\infty}q\|q\alpha\|\|q\gamma\|=0
    \qquad\text{for every $(\alpha,\gamma)\in\mathbb R^2$.}
\end{equation}
Most of the results known for \eqref{eq:plc} and
\eqref{eq:classical-lc} have close counterparts; see, in particular, the
discussion in \cite[Section~1.2]{EinsiedlerKleinbock2007}.

Let $\bad$ denote the set of badly approximable real
numbers. It is easy to see that every counterexample to
\eqref{eq:plc} belongs to $\bad$.  Similarly, if $(\alpha,\gamma)$
is a counterexample to \eqref{eq:classical-lc}, then both $\alpha$
and $\gamma$ belong to $\bad$. Indeed, these statements follow from
the inequalities $|q|_p\leq1$ and $\|qt\|\leq1/2$, respectively,
after decreasing the resulting positive constant to deal with
finitely many small values of $q$. Due to the Khintchine theorem,
$\bad$ has Lebesgue measure zero, therefore the exceptional sets in
both conjectures have Lebesgue measure zero.  Of course, the
conjectures assert that these exceptional sets are actually empty.

A celebrated theorem of Einsiedler, Katok and Lindenstrauss states
that the set of counterexamples to the classical Littlewood
conjecture has Hausdorff dimension zero
\cite{EinsiedlerKatokLindenstrauss2006}.  The corresponding result
for the $p$-adic Littlewood conjecture was proved by Einsiedler and
Kleinbock~\cite{EinsiedlerKleinbock2007}: the set of exceptions to
\eqref{eq:plc} has Hausdorff dimension zero. (In fact, both results
show that the set of exceptions is a countable union of sets of box
dimension zero).  These results show that a putative exceptional set
must be very small in a metric sense, but do not exclude its
existence.

In the present paper we approach this question from the opposite
direction. In Theorem~\ref{mainTHM1} below we prove that if the set
of counterexamples to the $p$-adic Littlewood conjecture is
non-empty, then it cannot be too small.  More precisely, it has
cardinality continuum and positive Hausdorff measure with respect to
the gauge function
\begin{equation} \label{eq:intro-h0}
    h_1(r)=\frac{1}{\log(1/r)}.
\end{equation}
In fact, we prove the stronger dynamical statement that the
conclusion holds for the closure of the orbit of any counterexample
under the map $x\mapsto px\pmod 1$.  Notice that this does not
contradict the results of
\cite{EinsiedlerKleinbock2007,EinsiedlerKatokLindenstrauss2006}:
Hausdorff dimension zero does not imply that Hausdorff measure with
respect to a logarithmic gauge such as \eqref{eq:intro-h0} vanishes.

There are several known restrictions on the complexity of a putative
counterexample.  Restrictions involving the continued fraction
expansion were obtained by Badziahin, Bugeaud, Einsiedler and
Kleinbock \cite{BadziahinBugeaudEinsiedlerKleinbock2015}. Our
argument uses a simple quantitative consequence of the same
Diophantine condition: a word of length $n$ in the base-$p$
expansion cannot return after fewer than $n-O(1)$ shifts. This
separation of occurrences gives a logarithmic upper bound for the
mass of every cylinder and hence a lower bound for the logarithmic
Hausdorff dimension of the orbit closure.

De Mathan and Teulie~\cite{deMathanTeulie2004} also introduced a
function field version of the $p$-adic Littlewood conjecture, widely
known as the $t$-adic Littlewood conjecture, which is intensively
studied nowadays. Let $\mathbb F_q$ be the finite field of $q$
elements and let $\mathbb F_q((t^{-1}))$ be the field of formal
Laurent series over $\mathbb F_q$. The conjecture asserts that
\begin{equation} \label{eq:intro-tlc}
 \inf_{0\ne Q\in\mathbb F_q[t]}
 |Q|\,|Q|_t\,|\langle Q\Theta\rangle|=0
 \qquad\text{for every }\Theta\in\mathbb F_q((t^{-1})),
\end{equation}
where
\[
    v_t(Q)=\max\{k\geq0:t^k\mid Q\},
    \qquad
    |Q|=q^{\deg Q},
    \qquad
    |Q|_t=q^{-v_t(Q)},
\]
and $\langle\cdot\rangle$ denotes the fractional part. Adiceam,
Nesharim and Lunnon~\cite{AdiceamNesharimLunnon2021} constructed an
explicit counterexample in characteristic three.  Garrett and
Robertson~\cite{GarrettRobertson2026} subsequently found
counterexamples in characteristics $5$, $7$ and $11$ with heavily
computational methods. Adiceam and
Badziahin~\cite{AdiceamBadziahin2025} then provided a purely
mathematical construction of counterexamples in all characteristics
congruent to $3$ modulo $4$.  Most recently, Lai and
Sprang~\cite{LaiSprang2026} gave an explicit construction over every
field of odd characteristic, in fact for the more general
$P(t)$-adic conjecture associated with an arbitrary irreducible
polynomial $P$.


In this paper we also treat the $t$-adic Littlewood conjecture and
show that (see Theorem~\ref{thm:t-conditional}) the set $E_q^{(t)}$
of counterexamples $Q$ to~\eqref{eq:intro-tlc}, if non-emtpy, has an
infinite positive Hausdorff measure with respect to the same gauge
function $h_1(r)$. Moreover, with help of the construction of Lai
and Sprang~\cite{LaiSprang2026} we refine the gauge function and
show that (see Theorem~\ref{thm:t-odd-strong}) over the field
$\mathbb F_q((t^{-1}))$ with odd $q$ one has
\begin{equation} \label{eq:intro-strong-t}
 \Haus^{h_{A_q}}(E_q^{(t)})=\infty
 \quad\text{ where } h_A(r) = \frac{1}{\log^A(1/r)},\quad
 A_q=\frac{q-1}{2}\log_2(q-1).
\end{equation}

\paragraph{Further questions.}
It would be very interesting to obtain an analogous lower bound for the set
of counterexamples to the classical Littlewood conjecture.  Another natural
problem is to strengthen the present result by replacing $h_1$ with a smaller
gauge.  For example, one may ask whether, for some $A>1$ (or perhaps for every
$A>1$), non-emptiness of the exceptional set implies
\begin{equation} \label{eq:future-gauge}
    \Haus^{h_A}(E_p)>0,
    \qquad
    h_A(r)=\frac{1}{(\log(1/r))^A}.
\end{equation}
The same question can be asked for the classical Littlewood conjecture.

Potential improvements of this type may also provide a route towards
the Littlewood conjectures themselves.  More precisely, suppose that
for some gauge $h$ one could prove both that non-emptiness of the
exceptional set implies positive $\Haus^h$-measure and, by
strengthening the methods of
\cite{EinsiedlerKatokLindenstrauss2006,EinsiedlerKleinbock2007},
that the same exceptional set has zero $\Haus^h$-measure.  The two
estimates would force the exceptional set to be empty.  Thus,
sufficiently strong and incompatible lower and upper estimates for a
putative exceptional set could lead to a proof of the corresponding
Littlewood conjecture.

\section{Logarithmic Hausdorff Measures of Exceptional Sets for the p-adic Littlewood Conjectures}
\subsection{Hausdorff measures, notation and the main result}

We briefly recall the definitions used below.  A gauge function (also called
a dimension function) is a continuous non-decreasing function
$h:[0,\infty)\to[0,\infty)$ such that $h(0)=0$ and $h(r)>0$ for $r>0$.
Let $M$ be a metric space. For $E\subseteq M$ and $\delta>0$, put
\begin{equation} \label{eq:hausdorff-content}
    \Haus^h_\delta(E)
    =\inf\left\{\sum_i h(\operatorname{diam}U_i):
      E\subseteq\bigcup_iU_i,\ \operatorname{diam}U_i\leq\delta\right\},
\end{equation}
where the infimum is taken over all countable covers of $E$ with
diameters at most $\delta$.  The Hausdorff $h$-measure of $E$ is
\begin{equation} \label{eq:hausdorff-measure}
    \Haus^h(E)=\lim_{\delta\to0}\Haus^h_\delta(E).
\end{equation}
For $h(r)=r^s$, this is the usual $s$-dimensional Hausdorff measure
$\Haus^s$, and
\begin{equation} \label{eq:hausdorff-dimension}
    \dim_{\mathrm H}E
    =\inf\{s\geq0:\Haus^s(E)=0\}
    =\sup\{s\geq0:\Haus^s(E)=\infty\}.
\end{equation}

For $s>0$, let
\[
    h_s(r)=\frac{1}{(\log(1/r))^s},
    \qquad 0<r<e^{-1},
\]
set $h_s(0)=0$, and extend $h_s$ away from zero in any continuous
non-decreasing way.  We call $\Haus^{h_s}$ the logarithmic
$s$-dimensional Hausdorff measure and define the logarithmic Hausdorff
dimension by
\begin{equation} \label{eq:log-dimension}
    \dim_{\log}E
    =\inf\{s>0:\Haus^{h_s}(E)=0\}.
\end{equation}

We shall use the following standard mass distribution principle for a general
gauge function \cite[Section~2.1]{BeresnevichVelani2006}.  We state it in the
form valid for an arbitrary metric space.
\begin{lemma}[Mass distribution principle]
Let $(M,d)$ be a metric space and let $\mu$ be a Borel probability measure
supported on $E\subseteq M$.
Suppose that there exist $c,r_0>0$ such that
\begin{equation}\label{eq:mass-distribution-hypothesis}
    \mu(U)\leq c\,h(\operatorname{diam}U)
\end{equation}
for every Borel set $U\subseteq M$ with
$\operatorname{diam}U\leq r_0$.  Then
$\Haus^h(E)\geq c^{-1}$.
\end{lemma}

%

We work on the circle $\T=\mathbb R/\mathbb Z$ and consider the continuous
map
\begin{equation} \label{eq:map}
    T_p(y)=py\pmod 1.
\end{equation}
For $x\in\T$ consider the closure
\begin{equation}  \label{eq:orbit-closure}
    \Fp(x)=\overline{\{T_p^k x:k\geq 0\}}.
\end{equation}
We also put
\begin{equation} \label{eq:mp}
    \mpc(x)=\liminf_{q\to\infty}q|q|_p\|qx\|
\end{equation}
and, for $\varepsilon>0$, define
\begin{equation} \label{eq:Lpepsilon}
    \Lp(\varepsilon)
    =\{x\in\T:\mpc(x)\geq\varepsilon\}.
\end{equation}
The exceptional set for \eqref{eq:plc} is therefore can be described
as
\begin{equation} \label{eq:exceptional-set}
    E_p=\{x\in\T:\mpc(x)>0\}.
\end{equation}
The set $F_p(x)$ satisfies the quantitative form of the
orbit-closure property which can be written as
\begin{equation} \label{eq:orbit-stability}
    x\in\Lp(\varepsilon)
    \quad\Longrightarrow\quad
    \Fp(x)\subseteq\Lp(\varepsilon).
\end{equation}
Statement \eqref{eq:orbit-stability} is proved in
\cite[Lemma~3.1]{BlackmanKristensenNorthey2024}.  A closely related reduction
from a lower bound for the liminf to a uniform lower bound over all positive
integers appears in \cite[Lemma~1]{Badziahin2026}.


Our main result about the counterexamples to the $p$-adic Littlewood
conjecture is

\begin{theorem}\label{mainTHM1}
Let $p$ be a prime and suppose that $x\in\T$ satisfies $\mpc(x)>0$.  Then
\begin{equation} \label{eq:main-conclusion}
    \Haus^{h_1}\bigl(\Fp(x)\bigr)>0.
\end{equation}
More generally, if $h$ is a dimension function such that
\begin{equation} \label{eq:gauge-positive}
    \liminf_{r\to0}h(r)\log(1/r)>0,
\end{equation}
then $\Haus^h(\Fp(x))>0$.  If
\begin{equation} \label{eq:gauge-infinite}
    h(r)\log(1/r)\longrightarrow\infty
    \qquad\text{as $r\to0$},
\end{equation}
then
\begin{equation} \label{eq:infinite-conclusion}
    \Haus^h\bigl(\Fp(x)\bigr)=\infty.
\end{equation}
\end{theorem}

The proof of Theorem~\ref{mainTHM1} will be presented in
Section~\ref{proof1} after establishing all the required preliminary
statements.

\begin{corollary}\label{cor_cardinal_cont}
If $E_p$ is non-empty, then
\begin{equation} \label{eq:exceptional-positive}
    \Haus^{h_1}(E_p)>0
\end{equation}
and $E_p$ has cardinality continuum.
\end{corollary}

\begin{proof}[Proof of Corollary~\ref{cor_cardinal_cont}]
Choose $x\in E_p$.  By \eqref{eq:orbit-stability}, $\Fp(x)\subseteq
E_p$, and \eqref{eq:exceptional-positive} follows from
\eqref{eq:main-conclusion}.  Since $h_1(r)\to0$ as $r\to0$, a
countable set has zero $h_1$-Hausdorff measure.  Hence $\Fp(x)$ is
uncountable.  It is also compact, and therefore has cardinality
continuum.  Since $\Fp(x)\subseteq E_p\subseteq\T$, the same is true
of $E_p$.
\end{proof}

%

\begin{corollary}\label{cor_dimlog_bound}
Under the assumptions of the Theorem~\ref{mainTHM1},
\begin{equation} \label{eq:log-dimension-bound}
    \dim_{\log}\Fp(x)\geq1.
\end{equation}
\end{corollary}

\subsection{A uniform Diophantine estimate on the orbit closure}\label{s3}

We first record the uniform estimate that will be used in the proof
of Theorem~\ref{mainTHM1}.  It is slightly different from
\eqref{eq:orbit-stability}: the constant may be smaller than
$\mpc(x)$, but the estimate holds for every denominator and is
uniform over the whole orbit closure.

\begin{lemma}\label{lemma_for_orbit_closure}
Suppose that $\mpc(x)>0$ and set
\begin{equation}  \label{eq:beta}
    \beta=\inf_{q\geq1}q|q|_p\|qx\|.
\end{equation}
Then $\beta>0$ and
\begin{equation}  \label{eq:uniform-bound}
    q|q|_p\|qy\|\geq\beta
    \qquad\text{for every $q\geq1$ and every $y\in\Fp(x)$.}
\end{equation}
\end{lemma}

\begin{proof}[Proof of Lemma~\ref{lemma_for_orbit_closure}]
The condition $\mpc(x)>0$ implies that $x$ is irrational.  Hence
every term in the infimum in \eqref{eq:beta} is positive.  All
sufficiently large terms are bounded from below by $\mpc(x)/2$, and
the remaining terms belong to a finite set. This proves that
$\beta>0$.

For every $q\geq1$ and $k\geq0$, we have
\begin{align} \label{eq:orbit-uniform}
    q|q|_p\|qT_p^k x\|
    &=q|q|_p\|qp^k x\| \notag\\
    &=p^kq\,|p^kq|_p\,\|p^kqx\| \notag\\
    &\geq\beta.
\end{align}
For each fixed $q$, the function $y\mapsto\|qy\|$ is continuous on
$\T$. Passing to the closure of the orbit in
\eqref{eq:orbit-uniform} proves \eqref{eq:uniform-bound}.
\end{proof}

\begin{corollary}\label{cor_no_rat_points} $\Fp(x)$ contains no
rational point. In particular, $0\notin\Fp(x)$.  We may therefore
identify $\Fp(x)$ with a compact subset of $(0,1)$, and every point
of $\Fp(x)$ has a unique base-$p$ expansion.
\end{corollary}
\begin{proof}[Proof of Corollary~\ref{cor_no_rat_points}]
Indeed, if $y=a/b \in \Fp(x)$ is rational, then $\|qy\|=0$ for
every positive multiple $q$ of $b$ contradicting
\eqref{eq:uniform-bound}.
\end{proof}

\subsection{Return times to cylinders}\label{s4}

For a word $w$ of length $n$ over $\{0,1,\ldots,p-1\}$, let $[w]$
denote the corresponding base-$p$ cylinder intersected with
$\Fp(x)$.  Thus $[w]$ is an interval of length $p^{-n}$, intersected
with the orbit closure. By the order of $[w]$ we call the length of
the corresponding word $w$.

\begin{lemma}\label{lemma_return_times}
Let $y\in\Fp(x)$ and $d\geq1$.  If $y$ and $T_p^d y$ belong to the
same base-$p$ cylinder of order $n$, then
\begin{equation}  \label{eq:return-bound}
    n<d+\log_p\frac{1}{\beta},
\end{equation}
where $\beta$ is defined in~\eqref{eq:beta}.
\end{lemma}

\begin{proof}[Proof of Lemma~\ref{lemma_return_times}]
Take $q=p^d-1$.  Since $|p^d-1|_p=1$, it follows from
\eqref{eq:uniform-bound} that
\begin{equation} \label{eq:return-start}
    \beta
    \leq(p^d-1)\bigl\|(p^d-1)y\bigr\|
    =(p^d-1)\|T_p^dy-y\|.
\end{equation}
If $y$ and $T_p^dy$ belong to the same cylinder of order $n$, then
\begin{equation} \label{eq:cylinder-distance}
    \|T_p^dy-y\|\leq p^{-n}.
\end{equation}
Combining \eqref{eq:return-start} and \eqref{eq:cylinder-distance},
we obtain
\begin{equation}   \label{eq:return-comparison}
    \beta<(p^d)p^{-n}=p^{d-n},
\end{equation}
which is equivalent to \eqref{eq:return-bound}.
\end{proof}


Put
\begin{equation} \label{eq:C}
    C=\left\lceil\log_p\frac{1}{\beta}\right\rceil.
\end{equation}
Then the return-time Lemma~\ref{lemma_return_times} implies that,
for every word $w$ of length $n>C$, the relative preimages
\begin{equation} \label{eq:disjoint-preimages}
    [w],\ T_p^{-1}[w],\ \ldots,\ T_p^{-(n-C)}[w]
\end{equation}
are pairwise disjoint subsets of $\Fp(x)$.  Indeed, if $z\in
T_p^{-i}[w]\cap T_p^{-j}[w]$ for some $0\leq i<j\leq n-C$, then
$T_p^iz$ and $T_p^jz$ belong to the same cylinder of order $n$.
Applying \eqref{eq:return-bound} with $d=j-i$ gives
\begin{equation} \label{eq:disjoint-contradiction}
    n<j-i+\log_p\frac{1}{\beta}
    \leq n-C+\log_p\frac{1}{\beta}
      \leq n,
\end{equation}
which is a contradiction.

\subsection{Proof of Theorem~\ref{mainTHM1}}\label{proof1}

Since $\Fp(x)$ is compact and forward $T_p$-invariant, it supports a
$T_p$-invariant Borel probability measure $\mu$.  For completeness,
one can take any weak limit point of the normalised counting
measures
\begin{equation} \label{eq:empirical-measures}
    \mu_N=\frac{1}{N}\sum_{k=0}^{N-1}\delta_{T_p^kx}.
\end{equation}
Such a limit is supported on $\Fp(x)$, and its invariance follows
from
\begin{equation} \label{eq:asymptotic-invariance}
    (T_p)_*\mu_N-\mu_N
    =\frac{1}{N}\bigl(\delta_{T_p^Nx}-\delta_x\bigr).
\end{equation}

By invariance, all the sets in \eqref{eq:disjoint-preimages} have
the same $\mu$-measure.  Their pairwise disjointness gives
\begin{equation}  \label{eq:cylinder-sum}
    (n-C+1)\mu([w])\leq1,
\end{equation}
and hence
\begin{equation} \label{eq:cylinder-measure}
    \mu([w])\leq\frac{1}{n-C+1}.
\end{equation}

Let $I\subseteq(0,1)$ be an interval of sufficiently small length
and choose $n$ so that
\begin{equation} \label{eq:choice-n}
    p^{-(n+1)}<|I|\leq p^{-n}.
\end{equation}
The interval $I$ meets at most three base-$p$ cylinders of order
$n$. It follows from \eqref{eq:cylinder-measure} that
\begin{equation} \label{eq:interval-frostman}
    \mu(I)\leq\frac{3}{n-C+1}
    \ll_{p,\beta}\frac{1}{\log(1/|I|)}
    =h_1(|I|).
\end{equation}
Every subset of $\mathbb R$ of diameter $r$ is contained in an
interval of length $r$.  Thus \eqref{eq:interval-frostman} gives
\begin{equation} \label{eq:frostman-bound}
    \mu(U)\ll_{p,\beta}h_1(\operatorname{diam}U)
\end{equation}
for every set $U$ of sufficiently small diameter.  The mass
distribution principle for the dimension function $h_1$
\cite[Section~2.1]{BeresnevichVelani2006} now yields
\eqref{eq:main-conclusion}.

If $h$ satisfies \eqref{eq:gauge-positive}, then $h(r)\gg h_1(r)$
for all sufficiently small $r$, and the same mass distribution
argument gives $\Haus^h(\Fp(x))>0$.  Finally, if
\eqref{eq:gauge-infinite} holds, then
\begin{equation} \label{eq:gauge-ratio}
    \frac{h_1(r)}{h(r)}\longrightarrow0.
\end{equation}
The standard comparison principle for Hausdorff measures
\cite[Lemma~1]{BeresnevichVelani2006}, applied to
\eqref{eq:main-conclusion} and \eqref{eq:gauge-ratio}, gives
\eqref{eq:infinite-conclusion}.  This completes the proof of
Theorem~\ref{mainTHM1}.

%

\begin{remark}
The Diophantine input is used only through the uniform estimate
\eqref{eq:uniform-bound}.  Its symbolic consequence is the linear
lower bound \eqref{eq:return-bound} for return times to cylinders.
Any stronger uniform return-time estimate would give a
correspondingly smaller logarithmic gauge by the same argument.
\end{remark}

\section{Logarithmic Hausdorff Measures of Exceptional Sets for the t-adic Littlewood Conjectures}
\subsection{The $t$-adic orbit-closure argument}

Let $q$ be a prime power.  For a non-zero Laurent series
\[
 \Gamma=\sum_{j\geq -j_0}c_jt^{-j}\in\mathbb F_q((t^{-1})),
 \qquad c_{-j_0}\ne0,
\]
put $\deg\Gamma=j_0$ and $|\Gamma|=q^{\deg\Gamma}$. By convention,
we also state $|0|=0$. Denote the fractional part of $\Gamma$ by
\[
 \langle\Gamma\rangle:=\sum_{j\geq1}c_jt^{-j}.
\]
The set
\[
 \Iq=t^{-1}\mathbb F_q[[t^{-1}]]
\]
is a compact metric space for the ultrametric
\begin{equation} \label{eq:t-metric}
 d(\Theta,\Phi)=|\Theta-\Phi|.
\end{equation}
If the first term at which $\Theta$ and $\Phi$ differ is for
$t^{-n}$, then $d(\Theta,\Phi)=q^{-n}$.  Thus the cylinders obtained
by prescribing the first $n$ coefficients of the elements in $\Iq$
are precisely the balls of diameter $q^{-(n+1)}$.

For $\Theta\in\Iq$ define
\begin{equation} \label{eq:mt-definition}
 \mt(\Theta)
 =\inf_{0\ne Q\in\mathbb F_q[t]}
 |Q|\,|Q|_t\,|\langle Q\Theta\rangle|,
 \qquad |Q|_t=q^{-v_t(Q)},
\end{equation}
and let
\begin{equation} \label{eq:t-exceptional-set}
 \Et=\{\Theta\in\Iq:\mt(\Theta)>0\}.
\end{equation}
Since every non-zero polynomial is of the form $t^kM$, where
$M(0)\ne0$, we have the equivalent expression
\begin{equation} \label{eq:mt-equivalent}
 \mt(\Theta)
 =\inf_{\substack{0\ne M\in\mathbb F_q[t]\\ k\geq0}}
 |M|\,|\langle Mt^k\Theta\rangle|.
\end{equation}
Allowing polynomials $M$ divisible by $t$ does not change the
infimum: such a factor can be transferred from $M$ to $t^k$, and
deleting it only decreases the first factor.  Thus $\Et$ is exactly
the set of counterexamples to \eqref{eq:intro-tlc}, represented by
their fractional parts.

Let
\begin{equation} \label{eq:t-shift}
 T\Theta=\langle t\Theta\rangle.
\end{equation}
In terms of the coefficient sequence, $T$ is the one-sided left
shift.  For $\Theta\in\Iq$ write
\begin{equation} \label{eq:t-orbit-closure}
 K(\Theta)=\overline{\{T^j\Theta:j\geq0\}}.
\end{equation}

\begin{theorem} \label{thm:t-conditional}
Let $q$ be any prime power.  If $\Theta\in\Et$, then
\begin{equation} \label{eq:t-orbit-positive}
 \Haus^{h_1}(K(\Theta))>0,
 \qquad h_1(r)=\frac1{\log(1/r)}.
\end{equation}
Moreover, if $\Et$ is non-empty, then
\begin{equation} \label{eq:t-exceptional-infinite}
 \Haus^{h_1}(\Et)=\infty.
\end{equation}
In particular, $\dim_{\log}\Et\geq1$ and $\Et$ has cardinality
continuum.
\end{theorem}

\begin{proof}[Proof of Theorem~\ref{thm:t-conditional}]
Put $c=\mt(\Theta)>0$.  The first observation is that
\begin{equation} \label{eq:t-shift-stability}
 \mt(T^j\Theta)\geq c\qquad(j\geq0).
\end{equation}
Indeed, $t^j\Theta-T^j\Theta$ is a polynomial, and hence
\[
 \langle Mt^kT^j\Theta\rangle
 =\langle Mt^{k+j}\Theta\rangle.
\]
The set $\{\Phi:\mt(\Phi)\geq c\}$ is closed, since it is an
intersection of closed conditions indexed by $M$ and $k$.  It
follows that
\begin{equation} \label{eq:t-closure-stability}
 K(\Theta)\subseteq\{\Phi\in\Iq:\mt(\Phi)\geq c\}.
\end{equation}

We next prove a lower bound for return times.  Suppose that $i<j$
and that $T^i\Theta$ and $T^j\Theta$ have the same first $n$
coefficients.  Taking $Q=t^j-t^i=t^i(t^{j-i}-1)$ in
\eqref{eq:mt-definition}, we obtain
\begin{align} \label{eq:t-return-computation}
 c
 &\leq |t^j-t^i|\,|t^j-t^i|_t
       |\langle(t^j-t^i)\Theta\rangle| \notag\\
 &=q^{j-i}|T^j\Theta-T^i\Theta|\\
 &\leq q^{j-i-n-1}. \notag
\end{align}
Consequently,
\begin{equation} \label{eq:t-return-time}
 j-i\geq n+1-\log_q(1/c).
\end{equation}

Consider the empirical measures
\[
 \mu_N=\frac1N\sum_{j=0}^{N-1}\delta_{T^j\Theta}
\]
and let $\mu$ be a weak limit point.  As in
\eqref{eq:asymptotic-invariance}, $\mu$ is $T$-invariant and
supported on $K(\Theta)$.  Let $[w]$ be a cylinder determined by a
word of length $n$. By \eqref{eq:t-return-time}, two visits of the
orbit to $[w]$ are separated by at least $n-C$ iterates, where
\begin{equation} \label{eq:t-return-constant}
 C=\left\lceil\log_q(1/c)\right\rceil.
\end{equation}
It follows either by counting visits in $\mu_N$, or by using the
disjoint preimages as in \eqref{eq:disjoint-preimages}, that
\begin{equation} \label{eq:t-cylinder-measure}
 \mu([w])\leq\frac1{n-C}\qquad(n>C).
\end{equation}
Every set $U\subseteq\Iq$ of sufficiently small positive diameter is
contained in a cylinder whose length is
$\log_q(1/\operatorname{diam}U)+O(1)$.  Hence
\begin{equation} \label{eq:t-frostman}
 \mu(U)\ll_{q,c}\frac1{\log(1/\operatorname{diam}U)}.
\end{equation}
The mass distribution principle then gives
\eqref{eq:t-orbit-positive}.

It remains to prove the stronger statement for the whole exceptional
set. For $a\in\mathbb F_q$ define the inverse branch
\begin{equation} \label{eq:inverse-branch}
 S_a(\Phi)=at^{-1}+t^{-1}\Phi.
\end{equation}
We claim that
\begin{equation} \label{eq:prefix-stability}
 \mt(S_a\Phi)\geq q^{-1}\mt(\Phi).
\end{equation}
For $k\geq1$ this follows from
\[
 \langle Mt^kS_a\Phi\rangle
 =\langle Mt^{k-1}\Phi\rangle.
\]
For $k=0$, put $z=\langle MS_a\Phi\rangle$.  Since $\langle
tz\rangle=\langle M\Phi\rangle$ and $|\langle tz\rangle|\leq q|z|$,
the same conclusion follows with the factor $q^{-1}$.  Iterating
\eqref{eq:prefix-stability}, every word $w\in\mathbb F_q^m$
determines a similarity $S_w$ of ratio $q^{-m}$ such that
\begin{equation} \label{eq:prefixed-copies}
 S_wK(\Theta)\subseteq\Et.
\end{equation}
For fixed $m$ these $q^m$ compact sets lie in distinct cylinders and
are therefore pairwise disjoint.  Moreover,
\begin{equation} \label{eq:slow-variation-h1}
 \frac{h_1(q^{-m}r)}{h_1(r)}\longrightarrow1
 \qquad(r\to0).
\end{equation}
It follows directly from the definition of the Hausdorff measure
that $\Haus^{h_1}(S_wK(\Theta))=\Haus^{h_1}(K(\Theta))$.  Therefore
\[
 \Haus^{h_1}(\Et)
 \geq q^m\Haus^{h_1}(K(\Theta))
\]
for every $m$, proving \eqref{eq:t-exceptional-infinite}.  The
remaining claims follow from identical to
Corollary~\ref{cor_cardinal_cont} and
Corollary~\ref{cor_dimlog_bound} proofs arguments.
\end{proof}

%
%

\subsection{A stronger unconditional result in odd characteristic}

We now exploit the explicit construction of Lai and Sprang~
\cite{LaiSprang2026} available in odd characteristic. Put
\begin{equation} \label{eq:d-and-Z}
 d=\frac{q-1}{2},
 \qquad
 Z=\{\zeta\in\mathbb F_q^\times:\zeta^d=-1\}.
\end{equation}
Thus $Z$ is the set of quadratic non-residues in $\mathbb F_q$ and
$|Z|=d$.  Define
\begin{equation} \label{eq:Aq}
 B=(q-1)^d,
 \qquad
 A_q=\log_2B=d\log_2(q-1)
      =\frac{q-1}{2}\log_2(q-1).
\end{equation}

For a collection of parameters
\begin{equation} \label{eq:parameter-space}
 \mathbf a=(a_{k,\zeta})_{k\geq0,\,\zeta\in Z},
 \qquad a_{k,\zeta}\in\mathbb F_q^\times,
\end{equation}
let
\begin{equation} \label{eq:parametrised-lambda}
 \Lambda_{\mathbf a}(t)
 =\sum_{k=0}^{\infty}\sum_{\zeta\in Z}
 a_{k,\zeta}\frac{t^{2^k}}{t^{2^{k+1}}-\zeta}
 \in\Iq.
\end{equation}
The series converges because every summand in the $k$th block has
absolute value $q^{-2^k}$.  Let
\begin{equation} \label{eq:Cq-family}
 \mathcal C_q=\{\Lambda_{\mathbf a}:\mathbf a
 \text{ satisfies }\eqref{eq:parameter-space}\}.
\end{equation}

\begin{theorem} \label{thm:t-odd-strong}
Let $q$ be an odd prime power.  Then
\begin{equation} \label{eq:family-exceptional}
 \mathcal C_q\subseteq
 \{\Theta\in\Iq:\mt(\Theta)\geq q^{-2(q-1)}\}
 \subseteq\Et.
\end{equation}
Furthermore,
\begin{equation} \label{eq:critical-family-measure}
 0<\Haus^{h_{A_q}}(\mathcal C_q)<\infty,
 \qquad
 h_{A_q}(r)=\frac1{(\log(1/r))^{A_q}},
\end{equation}
and
\begin{equation} \label{eq:odd-infinite-measure}
 \Haus^{h_A}(\Et)=\infty
 \qquad(0<A\leq A_q).
\end{equation}
In particular, $\dim_{\log}\Et\geq A_q$.
\end{theorem}

\begin{remark}
For $q=3$ the exponent in \eqref{eq:Aq} is $A_3=1$, while it is
strictly larger than one for every odd $q\geq5$.  For example,
\[
 A_5=4,\qquad A_7=3\log_2 6,\qquad A_9=12.
\]
The exponent one in Theorem~\ref{thm:t-conditional} comes from the
linear separation of return times in a single orbit.  The larger
exponent in Theorem~\ref{thm:t-odd-strong} comes from the
$B=(q-1)^d$ independent choices available at every dyadic scale in
\eqref{eq:parametrised-lambda}.
\end{remark}

We split the proof into the Diophantine and metric parts.  Write
$\sigma\mathbf a=(a_{k+1,\zeta})_{k\geq0,\zeta\in Z}$.  The series
in \eqref{eq:parametrised-lambda} satisfies the functional equation
\begin{equation} \label{eq:parametrised-functional-equation}
 \Lambda_{\mathbf a}(t)
 =\Lambda_{\sigma\mathbf a}(t^2)
  +\sum_{\zeta\in Z}a_{0,\zeta}\frac{t}{t^2-\zeta}.
\end{equation}

\begin{lemma}[Uniform Counterexample Lemma]\label{un_contrex_lemma_tadic}
For every $\mathbf a$ satisfying \eqref{eq:parameter-space},
\begin{equation} \label{eq:uniform-parametrised-bound}
 \inf_{\substack{0\ne M\in\mathbb F_q[t]\\m\geq0}}
 |M|\,|\langle Mt^m\Lambda_{\mathbf a}\rangle|
 \geq q^{-2(q-1)}.
\end{equation}
\end{lemma}

\begin{proof}[Proof of Lemma~\ref{un_contrex_lemma_tadic}]
Put
\begin{equation} \label{eq:D-polynomial}
 D(t)=t^{q-1}-1=t^{2d}-1=(t^d-1)(t^d+1).
\end{equation}
We first prove that there are no $\mathbf a$, $m\geq0$ and non-zero
$Q\in\mathbb F_q[t]$ such that
\begin{equation} \label{eq:key-violation}
 D\mid Q
 \qquad\text{and}\qquad
 |Q|\,|\langle Qt^m\Lambda_{\mathbf a}\rangle|<1.
\end{equation}
Suppose otherwise, and choose such a triple $(\mathbf a, m, Q)$ with
$\deg Q$ minimal among all parameter collections $\mathbf a$. Write
\begin{equation} \label{eq:Q-parity-decomposition}
 Q(t)=D(t)\widetilde Q(t),
 \qquad
 \widetilde Q(t)=\widetilde Q_0(t^2)+t\widetilde Q_1(t^2),
\end{equation}
and set
\begin{equation} \label{eq:Qi-definition}
 Q_i(t)=(t^d-1)\widetilde Q_i(t),\qquad i=0,1.
\end{equation}
Then
\begin{equation} \label{eq:Q-even-odd}
 Q(t)=Q_0(t^2)+tQ_1(t^2).
\end{equation}
Write
\begin{equation} \label{eq:degree-parities}
 \deg Q=2s+\varepsilon,
 \qquad
 m=2n+\delta,
 \qquad \varepsilon,\delta\in\{0,1\}.
\end{equation}
Since $D\mid Q$, we have $s\geq d$.  Moreover,
\begin{align} \label{eq:parity-degrees}
 \deg Q_\varepsilon&=s,
 &\deg\widetilde Q_\varepsilon&=s-d,\notag\\
 \deg Q_{1-\varepsilon}&\leq s-1+\varepsilon,
 &\deg\widetilde Q_{1-\varepsilon}&\leq s-d-1+\varepsilon.
\end{align}
In particular, $Q_\varepsilon\ne0$ and $\deg Q_\varepsilon<\deg Q$.

One can easily verify that
$$
 \frac{t^{2n+1+\delta}}{t^2-\zeta} -
 \frac{\zeta^{n+\delta}t^{1-\delta}}{t^2-\zeta} \in \mathbb F_q[t],
$$
therefore
\[
 \left\langle
 \frac{t^{2n+1+\delta}}{t^2-\zeta}
 \right\rangle
 =
 \frac{\zeta^{n+\delta}t^{1-\delta}}{t^2-\zeta}.
\]
We apply this formula to get
$$
\left\langle \frac{(Q_0(t^2) + tQ_1(t^2))t}{t^2-\zeta}\right\rangle
= \frac{tQ_0(\zeta) + \zeta Q_1(\zeta)}{t^2-\zeta}.
$$
By using this equation and separating the even and odd powers of
$t$, we obtain
\begin{align*}
 \left\langle Qt^m\Lambda_{\mathbf a}\right\rangle
 &=
 \left\langle
 Q_\delta(t^2)t^{2n+2\delta}
 \Lambda_{\sigma\mathbf a}(t^2)
 +
 \sum_{\zeta\in Z}
 a_{0,\zeta}
 \frac{\zeta^{n+1}Q_{1-\delta}(\zeta)}
      {t^2-\zeta}
 \right\rangle
 \\
 &\quad
 +t\left\langle
 Q_{1-\delta}(t^2)t^{2n}
 \Lambda_{\sigma\mathbf a}(t^2)
 +
 \sum_{\zeta\in Z}
 a_{0,\zeta}
 \frac{\zeta^{n+\delta}Q_\delta(\zeta)}
      {t^2-\zeta}
 \right\rangle .
\end{align*}


For a Laurent series $R$, we write $R=O(t^u)$ if $\deg R\leq u$. The
strict inequality in \eqref{eq:key-violation}, together with the
discreteness of the absolute value, implies that $\langle
Qt^m\Lambda_{\mathbf a}\rangle =O(t^{-\deg Q-1})$. Consequently,
\begin{align}
 \left\langle
 Q_\delta(t)t^{n+\delta}\Lambda_{\sigma\mathbf a}(t)
 +
 \sum_{\zeta\in Z}
 a_{0,\zeta}
 \frac{\zeta^{n+1}Q_{1-\delta}(\zeta)}
      {t-\zeta}
 \right\rangle
 &=O(t^{-s-1}),
 \label{eq:descent-first}
 \\
 \left\langle
 Q_{1-\delta}(t)t^n\Lambda_{\sigma\mathbf a}(t)
 +
 \sum_{\zeta\in Z}
 a_{0,\zeta}
 \frac{\zeta^{n+\delta}Q_\delta(\zeta)}
      {t-\zeta}
 \right\rangle
 &=O(t^{-s-1-\varepsilon}).
 \label{eq:descent-second}
\end{align}
These are the two basic descent identities.

For $\zeta\in Z$, equations \eqref{eq:d-and-Z} and
\eqref{eq:Qi-definition} give
\begin{equation} \label{eq:Qi-at-zeta}
 Q_i(\zeta)=-2\widetilde Q_i(\zeta).
\end{equation}
The degree bounds in \eqref{eq:parity-degrees} imply
\[
 \deg\widetilde Q_{1-\delta}\leq s-d,
 \qquad
 \deg\bigl(t^\delta\widetilde Q_\delta\bigr)
 \leq s-d+\varepsilon\delta.
\]
Hence, one can easily check that after multiplying
\eqref{eq:descent-first} by $\widetilde Q_{1-\delta}(t)$
and multiplying \eqref{eq:descent-second} by $t^\delta\widetilde
Q_\delta(t)$ both expressions are of order $O(t^{-d-1})$.

After each multiplication we take fractional parts.  We use the identity
\[
 \left\langle\frac{P(t)}{t-\zeta}\right\rangle
 =
 \frac{P(\zeta)}{t-\zeta},
 \qquad P\in\mathbb F_q[t].
\]
Together with \eqref{eq:Qi-at-zeta}, this gives
$$
\left\langle (t^d-1)\widetilde Q_0(t)\widetilde Q_1(t)
 t^{n+\delta}\Lambda_{\sigma\mathbf a}(t) - 2 \sum_{\zeta\in Z}
 a_{0,\zeta}
 \frac{\zeta^{n+1}\widetilde Q_{1-\delta}(\zeta)^2}
      {t-\zeta}\right\rangle = O(t^{-d-1}),
$$
$$
\left\langle (t^d-1)\widetilde Q_0(t)\widetilde Q_1(t)
 t^{n+\delta}\Lambda_{\sigma\mathbf a}(t) - 2\sum_{\zeta\in Z}
 a_{0,\zeta}
 \frac{\zeta^{n+2\delta}\widetilde Q_\delta(\zeta)^2}
      {t-\zeta}
 \right\rangle  =O(t^{-d-1}).
$$
%
By subtracting these two relations and dividing by the non-zero
scalar $2(-1)^\delta$, we obtain
\begin{equation} \label{eq:Vandermonde-input}
 \sum_{\zeta\in Z}
 a_{0,\zeta}\zeta^{n+\delta}
 \frac{\widetilde Q_0(\zeta)^2
       -\zeta\widetilde Q_1(\zeta)^2}
      {t-\zeta}
 =O(t^{-d-1}).
\end{equation}
Since
\[
 \frac1{t-\zeta}=\sum_{j\geq1}\zeta^{j-1}t^{-j},
\]
the first $d$ coefficients in \eqref{eq:Vandermonde-input} give
\begin{equation} \label{eq:Vandermonde-system}
 \sum_{\zeta\in Z}
 a_{0,\zeta}\zeta^{n+\delta+j-1}
 \bigl(\widetilde Q_0(\zeta)^2
       -\zeta\widetilde Q_1(\zeta)^2\bigr)=0,
 \qquad 1\leq j\leq d.
\end{equation}
We look at them as a system of linear equations in variables
$\widetilde{Q}_0(\zeta)^2 - \zeta\widetilde{Q}_1(\zeta)^2$ with the
corresponding matrix $(a_{0,\zeta}\zeta^{n+\delta+j-1})_{\zeta\in Z,
1\le j\le d}$. Its determinant equals
$$
\prod_{\zeta\in Z} a_{0,\zeta}\zeta^{n+\delta} \cdot \det
(\zeta^{j-1})_{\zeta\in Z, 1\le j\le d }\neq 0,
$$
because the matrix on the right hand side is the Vandermonde matrix
which is invertible. Therefore it has only a zero solution or in
other words
$$
 \widetilde Q_0(\zeta)^2
 =\zeta\widetilde Q_1(\zeta)^2
 \qquad(\zeta\in Z).
$$
Every $\zeta\in Z$ is a quadratic non-residue. The last equation
therefore forces
\begin{equation} \label{eq:both-vanish}
 \widetilde Q_0(\zeta)=\widetilde Q_1(\zeta)=0
 \qquad(\zeta\in Z).
\end{equation}

The elements of $Z$ are precisely the $d$ roots of $t^d+1$ in
$\mathbb F_q$.  These roots are distinct, since the characteristic of
$\mathbb F_q$ does not divide $d=(q-1)/2$.  Hence
\[
 t^d+1=\prod_{\zeta\in Z}(t-\zeta).
\]
We now deduce from \eqref{eq:both-vanish} that $t^d+1$ divides both
$\widetilde Q_0$ and $\widetilde Q_1$.  Together with
\eqref{eq:Qi-definition}, this gives
\begin{equation} \label{eq:D-divides-Qi}
 D\mid Q_0
 \qquad\text{and}\qquad
 D\mid Q_1.
\end{equation}

If $\varepsilon=\delta$, use \eqref{eq:descent-first}; otherwise use
\eqref{eq:descent-second}.  In either case \eqref{eq:both-vanish}
eliminates the rational sum and gives
\begin{equation} \label{eq:smaller-violation}
 \left|Q_\varepsilon\right|
 \left|\left\langle
 Q_\varepsilon t^{n+\varepsilon\delta}
 \Lambda_{\sigma\mathbf a}\right\rangle\right|<1.
\end{equation}
By \eqref{eq:D-divides-Qi}, the polynomial $Q_\varepsilon$ is
divisible by $D$, whereas \eqref{eq:parity-degrees} gives $0\leq\deg
Q_\varepsilon<\deg Q$.  This contradicts the global minimality of
$\deg Q$.  Thus \eqref{eq:key-violation} is impossible.

Finally, if \eqref{eq:uniform-parametrised-bound} failed, there
would be $M\ne0$ and $m\geq0$ for which the product in that formula
was smaller than $q^{-2(q-1)}=q^{-4d}$.  Taking $Q=DM$,
multiplication by $D$, whose degree is $2d$, can increase each of
the two absolute values by at most $q^{2d}$. We would obtain
\eqref{eq:key-violation}, a contradiction.  This proves the lemma.
\end{proof}

\begin{proof}[Proof of Theorem~\ref{thm:t-odd-strong}]
The Uniform Counterexample Lemma and \eqref{eq:mt-equivalent} prove
\eqref{eq:family-exceptional}.  It remains to estimate the size of
$\mathcal C_q$.

Expanding an individual summand gives
\begin{equation} \label{eq:individual-expansion}
 \frac{t^{2^k}}{t^{2^{k+1}}-\zeta}
 =\sum_{j\geq0}\zeta^j t^{-(2j+1)2^k}.
\end{equation}
Every positive integer $n$ has a unique representation
$n=(2j+1)2^k$. Hence the coefficient of $t^{-(2j+1)2^k}$ in
$\Lambda_{\mathbf a}$ is
\begin{equation} \label{eq:block-coefficients}
 \sum_{\zeta\in Z}a_{k,\zeta}\zeta^j.
\end{equation}
For $j=0,\ldots,d-1$, these $d$ coefficients uniquely determine the
whole vector $(a_{k,\zeta})_{\zeta\in Z}$, again by invertibility of
the Vandermonde matrix.  In particular, the map $\mathbf
a\mapsto\Lambda_{\mathbf a}$ is injective.

Give every parameter $a_{k,\zeta}$ the uniform distribution on
$\mathbb F_q^\times$, independently of all the others, and let $\nu$
be the push-forward of this product probability measure to $\mathcal
C_q$.  There are $B=(q-1)^d$ choices for the vector at each fixed
scale $k$.  If the first $N$ Laurent coefficients are prescribed,
then all parameters $a_{k,\zeta}$ with $\zeta\in Z$ and
\begin{equation} \label{eq:visible-blocks}
 (2d-1)2^k=(q-2)2^k\leq N
\end{equation}
are determined by \eqref{eq:block-coefficients}.  Therefore every
cylinder of length $N$ has $\nu$-measure at most
\begin{equation} \label{eq:parameter-cylinder-bound}
(q-1)^{-d\log_2\left\lfloor\frac{N}{q-1}\right\rfloor}\ll_q
C_qN^{-A_q}.
\end{equation}
Since $N\asymp_q\log(1/r)$ for a cylinder of diameter $r$, every
Borel set $U$ of sufficiently small diameter satisfies
\begin{equation} \label{eq:parameter-frostman}
 \nu(U)\ll_q
 \frac1{(\log(1/\operatorname{diam}U))^{A_q}}.
\end{equation}
The mass distribution principle proves the strict lower bound in
\eqref{eq:critical-family-measure}.

For the upper bound, fix the vectors $(a_{k,\zeta})_{\zeta\in Z}$ at
scales $0\le k\le K-1$. All remaining blocks in
\eqref{eq:parametrised-lambda} start at an index at least $2^K$.
Thus each of the $B^K$ choices of the first $K$ vectors determines a
set of diameter at most $q^{-2^K}$. Consequently,
\begin{equation} \label{eq:critical-cover}
 \Haus^{h_{A_q}}_{q^{-2^K}}(\mathcal C_q)
 \leq B^Kh_{A_q}(q^{-2^K})
 =\frac{B^K}{(2^K\log q)^{A_q}}
 =(\log q)^{-A_q},
\end{equation}
where we used $2^{A_q}=B$.  This proves the upper bound in
\eqref{eq:critical-family-measure}.

Finally, apply the inverse branches \eqref{eq:inverse-branch} to
$\mathcal C_q$.  By \eqref{eq:prefix-stability}, every finite
prefixed copy $S_w\mathcal C_q$ is contained in $\Et$.  For fixed
word length $m$ these copies are pairwise disjoint, and
\begin{equation} \label{eq:slow-variation-hA}
 \frac{h_{A_q}(q^{-m}r)}{h_{A_q}(r)}\longrightarrow1
 \qquad(r\to0).
\end{equation}
Hence each copy has the same positive $h_{A_q}$-Hausdorff measure as
$\mathcal C_q$.  Letting $m$ tend to infinity proves
$\Haus^{h_{A_q}}(\Et)=\infty$.  If $0<A<A_q$, then $h_A(r)\geq
h_{A_q}(r)$ for all sufficiently small $r$, so the same infinite
measure conclusion follows by comparison.  This completes the proof.
\end{proof}


\begin{thebibliography}{99}

\bibitem{AdiceamBadziahin2025}
F.~Adiceam and D.~Badziahin,
\newblock On the $P(t)$-adic Littlewood conjecture in characteristics
$\ell\equiv3\pmod4$,
\newblock preprint, 2025,
\newblock \href{https://arxiv.org/abs/2509.12826}{arXiv:2509.12826}.

\bibitem{AdiceamNesharimLunnon2021}
F.~Adiceam, E.~Nesharim and F.~Lunnon,
\newblock On the $t$-adic Littlewood conjecture,
\newblock \emph{Duke Mathematical Journal} \textbf{170} (2021), no.~10,
2371--2419,
\newblock \href{https://doi.org/10.1215/00127094-2020-0077}
{doi:10.1215/00127094-2020-0077}.

\bibitem{Badziahin2026}
D.~Badziahin,
\newblock Estimating the lower limit in the $p$-adic Littlewood conjecture,
\newblock \emph{Experimental Mathematics} (2026), 1--6,
\newblock \href{https://doi.org/10.1080/10586458.2025.2595008}
{doi:10.1080/10586458.2025.2595008}.

\bibitem{BadziahinBugeaudEinsiedlerKleinbock2015}
D.~Badziahin, Y.~Bugeaud, M.~Einsiedler and D.~Kleinbock,
\newblock On the complexity of a putative counterexample to the $p$-adic
Littlewood conjecture,
\newblock \emph{Compositio Mathematica} \textbf{151} (2015), 1647--1662,
\newblock \href{https://doi.org/10.1112/S0010437X15007393}
{doi:10.1112/S0010437X15007393}.

\bibitem{BeresnevichVelani2006}
V.~Beresnevich and S.~Velani,
\newblock A mass transference principle and the Duffin--Schaeffer conjecture
for Hausdorff measures,
\newblock \emph{Annals of Mathematics} \textbf{164} (2006), 971--992,
\newblock \href{https://doi.org/10.4007/annals.2006.164.971}
{doi:10.4007/annals.2006.164.971}.

\bibitem{BlackmanKristensenNorthey2024}
J.~Blackman, S.~Kristensen and M.~J.~Northey,
\newblock A note on the base-$p$ expansions of putative counterexamples to
the $p$-adic Littlewood conjecture,
\newblock \emph{Expositiones Mathematicae} \textbf{42} (2024), no.~3,
Paper No.~125548, 16 pp.,
\newblock \href{https://doi.org/10.1016/j.exmath.2024.125548}
{doi:10.1016/j.exmath.2024.125548}.

\bibitem{EinsiedlerKatokLindenstrauss2006}
M.~Einsiedler, A.~Katok and E.~Lindenstrauss,
\newblock Invariant measures and the set of exceptions to Littlewood's
conjecture,
\newblock \emph{Annals of Mathematics} \textbf{164} (2006), 513--560,
\newblock \href{https://doi.org/10.4007/annals.2006.164.513}
{doi:10.4007/annals.2006.164.513}.

\bibitem{EinsiedlerKleinbock2007}
M.~Einsiedler and D.~Kleinbock,
\newblock Measure rigidity and $p$-adic Littlewood-type problems,
\newblock \emph{Compositio Mathematica} \textbf{143} (2007), 689--702,
\newblock \href{https://doi.org/10.1112/S0010437X07002801}
{doi:10.1112/S0010437X07002801}.

\bibitem{GarrettRobertson2026}
S.~Garrett and S.~Robertson,
\newblock Counterexamples to the $p(t)$-adic Littlewood conjecture over small
finite fields,
\newblock \emph{Mathematics of Computation} \textbf{95} (2026), no.~360,
1961--1986,
\newblock \href{https://arxiv.org/abs/2405.14454}{arXiv:2405.14454}.

\bibitem{LaiSprang2026}
L.~Lai and J.~Sprang,
\newblock On the $P(t)$-adic Littlewood conjecture in odd characteristics,
\newblock preprint, 2026,
\newblock \href{https://arxiv.org/abs/2606.00633}{arXiv:2606.00633}.

\bibitem{deMathanTeulie2004}
B.~de Mathan and O.~Teuli\'e,
\newblock Probl\`emes diophantiens simultan\'es,
\newblock \emph{Monatshefte f\"ur Mathematik} \textbf{143} (2004), 229--245,
\newblock \href{https://doi.org/10.1007/s00605-003-0199-y}
{doi:10.1007/s00605-003-0199-y}.

\bibitem{Robertson2026}
S.~Robertson,
\newblock Combinatorics on number walls and the $P(t)$-adic Littlewood
conjecture,
\newblock \emph{Mathematika} \textbf{72} (2026), no.~1, e70064,
\newblock \href{https://doi.org/10.1112/mtk.70064}
{doi:10.1112/mtk.70064}.

\end{thebibliography}
\end{document}